\documentclass[12pt,a4paper]{amsart}
\calclayout
\numberwithin{equation}{section}
\allowdisplaybreaks
\theoremstyle{plain}

\newtheorem{theorem}{Theorem}[section]
\newtheorem{lemma}[theorem]{Lemma}
\newtheorem{prop}[theorem]{Proposition}

\newtheorem{corollary}[theorem]{Corollary}
\newtheorem{remark}{Remark}

\numberwithin{equation}{section}

\usepackage{enumerate}

\usepackage{amssymb}
\usepackage{mathtools}
\mathtoolsset{showonlyrefs}
\usepackage{mathrsfs}
\usepackage{comment}
\usepackage{hyperref}
\usepackage{xcolor}

\def\P{\mathbb{P} }

\def\R{\mathbb{R} }
\def\d{\mathrm{d}}
\def\D{\mathcal{D} }
\def\E{\mathbb{E} }
\def\e{\mathcal{E}}
\def\F{\mathcal{F}}
\def\1{\mathbf{1}}
\def\d{\mathrm{d}}
\def\bP{\mathbf{P}}
\def\bE{\mathbf{E}}
\def\q{q}

\begin{document}
	\title[Killed BBM with drift]{From killed BBM with drift to
	BBM: the extremal process
	}
	\author[Y.-X. Ren, R. Song and F. Yang]{Yan-Xia Ren, Renming Song and Fan Yang}
	\address{Yan-Xia Ren\\ LMAM School of Mathematical Sciences \& Center for
		Statistical Science\\ Peking University\\ Beijing 100871\\ P. R. China}
	\email{yxren@math.pku.edu.cn}
	\thanks{}
	\address{Renming Song\\ Department of Mathematics\\ University of Illinois Urbana-Champaign \\ Urbana \\ IL 61801\\ USA}
	\email{rsong@illinois.edu}
	\address{Fan Yang\\ School of Mathematical Sciences \\ Beijing University of Posts and Telecommunications\\ Beijing 100876 \& Key Laboratory of Mathematics and Information Networks (Beijing University of Posts and Telecommunications) \\ Ministry of Education \\  P. R. China}
	\email{fan-yang@bupt.edu.cn}
	\thanks{}

	\begin{abstract}
       In this paper, we study the asymptotic behavior of the extremal process of        a standard one-dimensional branching Brownian motion (BBM) with drift $-\rho>-\sqrt2$ and  absorbing barrier at level $-x$. We prove  that the two-dimensional point process, with first component being the extremal process of the BBM and the second component being the running minimum of the BBM with drift, converges weakly to a decorated Poisson point process (DPPP) on $\mathbb{R} \times [0, \infty)$.  This framework allows us
       to supply a common limit coupling of
       the extremal processes for different barrier levels and to explicitly derive the limit, as $t\to\infty$, of the extremal process of the killed BBM killed at level $-x$, demonstrating that the  double limit, when $t\to\infty$ first and then $x\to\infty$, of the extremal process of the BBM killed at level $-x$
       coincides with the limit of the        extremal process of (un-killed) BBM.
	\end{abstract}
	
	\maketitle
	\section{Introduction}\label{Sec1}
	\subsection{Background}\label{Sec1.1}
	A classical branching Brownian motion (BBM) in $\R$ starts with a single particle at the origin, which moves as a 1-dimensional standard Brownian motion $(\{B(t)\}_{ t\geq 0}, \mathbf{P})$. After an exponential time with parameter $1$, independent of the spatial motion, it dies and gives birth to a random number $L$ 	offspring, with ${\rm P}(L=k)=p_k$, $k=0, 1, \dots$. Starting from their birth place, each of these particles evolves independently, according to the same spatial motion and following the same branching mechanism. We denote by $N_t$ the collection of particles alive at time $t$. For any $u\in N_t$ and $s\le t$, let $X_u(s)$ be the position of $u$ at time $s$ or the position of the ancestor of $u$ at time $s$.
	The process
    \begin{equation}\label{def_X_t}
    	\mathbb{X}_t:=\sum_{u\in N_t}\delta_{X_u(t)}
    \end{equation}
    is called a branching Brownian motion.
    Define $\F_t := \sigma(\mathbb{X}_s:s\leq t)$.
    We will  use $\P$ to denote the law of this BBM and use $\E$ to denote expectation with respect to $\P$.

    In this paper, we will always assume that
    \begin{equation}
    	L\geq 1, \quad {\rm E} L = 2, \quad {\rm E} L^2<\infty.
    \end{equation}
    When $p_2=1$, $\mathbb{X}$ is called a dyadic BBM.
	
    Let $M_t:= \max_{u\in N_t} \{X_u(t)\}$ be the maximum of the BBM $\mathbb{X}$ at time $t$.
    The asymptotic behavior of $M_t$ has attracted considerable interest for decades. Bramson \cite{Bramson78, Bramson83} proved that
	\begin{equation}\label{eq:Mt_speed}
		\lim_{t\to\infty}\P(M_t\le m_t+z)=\lim_{t\to \infty} u(t,m_t+z)=w(z),\quad z\in \R,
	\end{equation}
	where
	\begin{equation}
		m_t:=\sqrt{2}t-\frac{3}{2\sqrt{2}}\log t,
	\end{equation}
	and $w$ solves the ordinary differential equation
	\begin{equation*}
		\frac{1}{2}w''+\sqrt{2}w'+f(w)-w=0,
	\end{equation*}
    where $f(z)=\sum^{\infty}_{n=0}p_n z^n$. 	Lalley and Sellke \cite{Lalley87} provided the following representation of $w$ for dyadic BBM
	\begin{equation}\label{travelling}
		w(z):=\E\left[e^{-C_* e^{-\sqrt{2} z}Z_{\infty}}\right],
	\end{equation}
	where $C_*$ is a positive constant and $Z_{\infty}$ is the limit of the derivative martingale
	\begin{equation}\label{def_derivative}
		Z_t = \sum_{u\in N_t} (\sqrt{2}t - X_u(t)) e^{\sqrt{2}(X_u(t)-\sqrt{2}t) }.
	\end{equation}
	In the case of general offspring distributions, \eqref{travelling} was proved by Kyprianou \cite{Kyprianou04}.
	Moreover, it is known that
	\begin{equation}\label{eq:w_asymptotic}
		\lim_{x\rightarrow\infty} \frac{1-w(x)}{C_*xe^{-\sqrt{2}x}} = 1.
	\end{equation}

    Study on more complete structure of the BBM seen from its tip was independently initiated
    by  A\"{i}d\'{e}kon, Berestycki, Brunet and Shi \cite{ABBS13} and Arguin, Bovier and Kistler \cite{ABK13}. Define the  extremal process of BBM $\mathbb{X}_t$ by
	\begin{equation}
		\mathcal{E}_t:=\sum_{u\in N_t} \delta_{X_u(t)-m_t}.
	\end{equation}
    It was shown in \cite{ABBS13} and \cite{ABK13} that, as $t\to\infty$,  $\mathcal{E}_t$ converges in law to a random shifted decorated Poisson point process (DPPP). A DPPP $\e$ is determined by two components: an intensity measure $\mu$, which is a (random) measure on $\R$, and a decoration process. Conditioned on $\mu$, let $\sum_{i} \delta_{p_i}$ be a Poisson point process with intensity $\mu$, and let $\{\sum_j \delta_{d^i_j}\}$ be a family of independent point processes with law $\D$. Then $\e = \sum_{i,j} \delta_{p_i+d^i_j}$ is a DPPP with intensity $\mu$ and decoration $\D$, denoted by DPPP$(\mu, \D)$. A\"{i}d\'{e}kon et al. \cite{ABBS13} and Arguin et al. \cite{ABK13} proved  that
	\begin{equation}
		\lim_{t\rightarrow\infty}
		\mathcal{E}_t
		= \mathrm{DPPP}\left(\sqrt{2}C_*Z_{\infty}e^{-\sqrt{2}x}\d x, \D^{\sqrt{2}}\right) \mbox{ in law },
	\end{equation}
	with  $C_*$ as in \eqref{travelling} and
	\begin{equation}\label{def_D}
		\D^{\sqrt{2}}(\cdot) := \lim_{t\rightarrow\infty} \P\left( \sum_{u\in N_t} \delta_{X_u(t)-M_t} \in \cdot \, \Big| M_t \geq \sqrt{2}t \right).
	\end{equation}

    In this paper, we study branching Brownian motion with drift.
    For any $\rho\in \R$, the process
    \begin{equation}\label{def_X^rho_t}
         \mathbb{X}^{-\rho}_t:=\sum_{u\in N_t}\delta_{X_u(t)-\rho t}
    \end{equation}
    is a branching Brownian motion with
    drift $-\rho$.

    Let $x>0$ and define
    \begin{equation}
		\widetilde{N}_t^{x, -\rho} := \{u\in N_t: \forall s\leq t, X_u(s)-\rho s > -x  \}.
	\end{equation}
    $\widetilde{N}_t^{x, -\rho}$ is the subset of $N_t$ consisting of all the particles
    in BBM with drift $-\rho$ that have not hit $-x$ by time $t$.
    The process
    \begin{equation}
    	\mathbb{Y}^{x, -\rho}_t:=\sum_{u\in \widetilde{N}_t^{x, -\rho}}\delta_{X_u(t)-\rho t}
    \end{equation}
    is called a branching Brownian motion with drift $-\rho$ killed upon reaching $-x$.
    According to Kesten \cite{Kesten78}, this
    process dies out almost surely when $\rho \ge \sqrt{2}$ while it survives with positive probability when $\rho < \sqrt{2}$. Therefore, $\rho = \sqrt{2} $ is the critical drift separating the supercritical case $\rho < \sqrt{2}$ and the subcritical $\rho > \sqrt{2}$.

    In this paper, for $u\in N_t$, $X_u(t)$ is always the position of particle $u$ at time $t$ in a BBM (without drift).
    Note that for $u\in N_t$,
    \begin{equation}\label{e:rs-rel}
    	(X_u(t) - \rho t) - (m_t - \rho t) = X_u(t) - m_t.
    \end{equation}
    Thus,
    the center function for  $\mathbb{X}^{-\rho}_t$
    is $m^{-\rho}_t:=m_t-\rho t$. It is proved in
    \cite{YZ24}
    that the center function for $\mathbb{Y}^{x,-\rho}_t$ is also $m^{-\rho}_t$.

    The extremal process of $\mathbb{X}_t$
	\begin{equation}
		\sum_{u\in N_t}  \delta_{X_u(t)-m_t}
	\end{equation}
	is also the extremal process of $\mathbb{X}^{-\rho}_t$, the BBM with drift $-\rho$,
    and
    \begin{equation}
    \sum_{u\in \widetilde{N}_t^{x, -\rho}}\delta_{
	X_u(t)-m_t}
    \end{equation}
    is the extremal process of $\mathbb{Y}^{x, -\rho}$.

	It is natural to ask what the limiting behavior of the limit of the extremal process of $\mathbb{Y}^{x, -\rho}$ is as $x\to\infty$, and whether this limit coincides with the limit of extremal process of BBM. To accomplish this, we will study the convergence of the following point process
	\begin{align}
		\sum_{u\in N_t}  \delta_{\left(X_u(t)-m_t,-\inf\limits_{s\leq t} \{X_u(s)-\rho s\} \right)}.
	\end{align}

	\subsection{Main result}
	To state our main result, we first
	introduce a truncated version of the derivative martingale. For $0\leq s\leq t$ and $x > 0$, we define
	\begin{equation}\label{def_truncated_derivative_rho}
        Z^{-\rho}(x,s,t) := \sum_{u\in N_t}
		\left( \sqrt{2}t-X_u(t)\right) e^{\sqrt{2}(X_u(t)-\sqrt{2}t)} \1_{\left\{ \inf_{r\leq s} \{X_u(r)-\rho r\} >-x
		\right\} }.
	\end{equation}
      Note that $Z^{-\rho}(\infty,s, t)=Z_t$ for any $t\geq s\geq 0$. Since $\{Z_t, t\geq 0\}$ is the derivative  martingale, and $\1_{\left\{ \inf_{r\leq s} \{X_u(r)-\rho r\} >-x 		\right\} }$ does not depend on $t$, for fixed $x>0$ and $s>0$, $\{Z^{-\rho}(x,s,t), t>s; \P\}$ is a martingale. We call $Z^{-\rho}(x,s,t)$ a  ``truncated" derivative martingale. We have the following result for this ``truncated" derivative martingale.

	\begin{prop}\label{prop_truncated_derivative}
    Suppose $\rho<\sqrt{2}$. For any $x>0$, the limit  $\lim_{s\rightarrow\infty}\lim_{t\rightarrow\infty} Z^{-\rho}(x,s,t)$ exists $\P$-almost surely. Let
		\begin{equation}\label{def_Z(x)_rho}
			Z^{-\rho}(x) := \lim_{s\rightarrow\infty} \lim_{t\rightarrow\infty} Z^{-\rho}(x,s,t).
		\end{equation}
	Then $Z^{-\rho}(x)$ is non-decreasing in $x$,
	and $0\leq Z^{-\rho}(x) \leq Z_{\infty}$.
	\end{prop}

   We use $Z^{-\rho}(\d x)$ to denote the  measure on $(0, \infty)$ such that
   \begin{equation}\label{def-LS}
       Z^{-\rho}((0, x))=Z^{-\rho}(x),\quad \forall x>0.
   \end{equation}

	\begin{theorem}\label{thrm1}
		For any $\rho<\sqrt{2}$, we have that,  under $\P$
		\begin{align}
			\lim_{t\rightarrow\infty} \sum_{u\in N_t} \delta_{\left(X_u(t)-m_t,-\inf\limits_{s\leq t} \{X_u(s)-\rho s\} \right)} = \mathrm{DPPP}(\sqrt{2}C_*e^{-\sqrt{2}y}\d y \times \d Z^{-\rho}, \widetilde{\D}^{\sqrt{2}}) \mbox{ in law, }
		\end{align}
		where $C_*$ is as in \eqref{travelling},
		      and $\widetilde{\D}^{\sqrt{2}}$ is the law of $\sum_j \delta_{(d_j,0)}$ with $\sum_j \delta_{d_j}$ having the law $\D^{\sqrt{2}}$ given by \eqref{def_D}.
	\end{theorem}

    For any $0\le s\le t$ and $u\in N_t$, define
	\begin{equation}\label{def_gmma_u}
		\gamma_u(s) := \inf_{r\leq s}\left( X_u(r)-\rho r \right).
	\end{equation}
	For integers $n\ge1$ and $u\in N_n$, put
    	\begin{equation}\label{def-mu-n}
    	    Z_u(n)=e^{\sqrt{2} X_u(n)-2n}Z_\infty^{(u)},\quad     	\mu_n=\sum_{u\in N_n}Z_u(n)\delta_{-\gamma_u(n)},
    	\end{equation}
    	where $Z_\infty^{(u)}$ is the derivative martingale limit associated with the subtree rooted at $(n,X_u(n))$ after resetting its time and space origins, that is,
    	\begin{equation}
    		Z_{\infty}^{(u)} = \lim_{s\rightarrow\infty} \sum_{v\in N_{n+s},\, v\succ u} (\sqrt{2}s -  (X_{v}(n+s)-X_u(n))) e^{\sqrt{2}(X_{v}(n+s)-X_u(n)) - 2s}, \quad \P\mbox{-a.s.}
    	\end{equation}

	The following result gives more information about the random measure $Z^{-\rho}(\cdot)$.
	\begin{prop}\label{prop:proportion}
		Suppose  $\rho<\sqrt{2}$.
		\begin{itemize}
			\item[(i)] The measures  $\mu_n$, defined in  \eqref{def-mu-n}, converge almost surely in total variation to  $Z^{-\rho}(\d x)$ as $n\to\infty$. Moreover, $Z^{-\rho}(\d x)$  is an  atomic random measure on $(0,\infty)$ with total mass $Z_\infty$.
			
			\item[(ii)]	 For any $x>0$, $Z^{-\rho}(\{x\}) = 0$ almost surely.

			\item[(iii)] If $\rho \leq -\sqrt{2}$, then $Z^{-\rho}(\d x)$ has almost surely only finitely many atoms. If $-\sqrt{2}<\rho<\sqrt{2}$, then $Z^{-\rho}(\d x)$ has almost surely countably infinitely many atoms and unbounded support.
		\end{itemize}
	\end{prop}
	
    As a direct corollary of Theorem \ref{thrm1}, we obtain the following result, which
    answers the question mentioned at the end of Section
    \ref{Sec1.1}.
	\begin{corollary}\label{cor_extremal}
        Suppose $\rho<\sqrt{2}$.

        $\mathrm{(i)}$ Under $\P$,
		\begin{equation}
			\lim_{t\rightarrow\infty} \sum_{u\in N_t} \delta_{X_u(t)-m_t}\mathbf{1}_{\left\{\inf_{s\leq t} \{X_u(s)-\rho s\} >-x\right\}} = \mathrm{DPPP}(\sqrt{2}C_*Z^{-\rho}(x)e^{-\sqrt{2}y}\d y, \D^{\sqrt{2}}) \mbox{ in law.}
		\end{equation}

        $\mathrm{(ii)}$ Under $\P$,
		\begin{equation}\label{eq:extremal_BBMAB-BBM}
			\lim_{x\rightarrow\infty}\lim_{t\rightarrow\infty} \sum_{u\in N_t} \delta_{X_u(t)-m_t}\mathbf{1}_{\left\{\inf_{s\leq t} \{X_u(s)-\rho s\} >-x\right\}} = \mathrm{DPPP}(\sqrt{2}C_*Z_{\infty}e^{-\sqrt{2}y}\d y, \D^{\sqrt{2}}) \mbox{ in law,}
		\end{equation}
	    that is, the double limit, when $t\to\infty$ first and then $x\to\infty$, of the extremal process of killed BBM with drift is equal to the limit of the extremal process of classical BBM  as $t\to\infty$.
    \end{corollary}
	
	Theorem \ref{thrm1} and Corollary \ref{cor_extremal} (i) give a common limit coupling of the extremal
	processes for different barrier levels.

	\begin{remark}
		Let
		\begin{equation}\label{def_X^xrho_t}
			\mathbb{X}^{x, -\rho}_t:=\sum_{u\in N_t}\delta_{x+X_u(t)-\rho t}
		\end{equation}
		denote a branching Brownian motion with drift $-\rho$ starting from $x$. Then
		\begin{equation}
			\sum_{u\in \widetilde{N}_t^{x, -\rho}} \delta_{x+X_u(t)-m_t}
		\end{equation}
		is the extremal process of this BBM with absorbing barrier at the origin.
		Yang and Zhu \cite{YZ24} showed that, under $\P$,
		\begin{equation}
			\lim_{t\rightarrow\infty}
			\sum_{u\in \widetilde{N}_t^{x, -\rho}}\delta_{x+X_u(t)-m_t}
			= \mathrm{DPPP}\left(\sqrt{2}C_* 		Z_{\infty}^{x, -\rho}
			e^{-\sqrt{2}y}\d y, \D^{\sqrt{2}}\right) \mbox{ in law},
		\end{equation}
		where
		\begin{equation}
			 Z_{\infty}^{x, -\rho}
			 = \lim_{t\rightarrow\infty}  \sum_{u\in \widetilde{N}_t^{x, -\rho}} (\sqrt{2}t - X_u(t)-x) e^{\sqrt{2}(x+X_u(t)-\sqrt{2}t) }.
		\end{equation}
		Actually, it can be shown that
		\begin{equation}
			(Z_{\infty}^{x, -\rho}, \P) \overset{d.}{=} (e^{\sqrt{2}x}Z^{-\rho}(x), \P).
		\end{equation}
		Therefore,
        Corollary \ref{cor_extremal}
        $\mathrm{(i)}$
        is essentially the main result of \cite{YZ24}, with the only difference being that the starting point and the absorbing barrier in \cite{YZ24} are both  shifted upward by $x$.
	\end{remark}

	\section{Some properties of branching Brownian motion}\label{Sec2}
	Throughout this paper we use $\{B_t, t\geq 0; \mathbf{P} \}$ to denote a standard Brownian motion starting from the origin. Expectation with respect to $\mathbf{P}$ will be denoted by $\mathbf{E}$. Then, $\{x+B_t, t\geq 0; \bP \}$ is a standard Brownian motion starting from $x$.

	Let $\{\mathcal{F}_t^B: t\geq 0\}$ be the natural filtration of $\{B_t, t\geq 0\}$.
	For BBMs, the many-to-one lemma (see, for example, \cite[Theorem 8.5]{HH09} and \cite{Harris17}) is fundamental. We will need two versions of the many-to-one lemma, one version being the stopping line version which can be found in \cite[\S 2.3]{Maillard12}.

	Let $D$ be a time-space domain. We observe each particle at the first, in the family history, exit time from $D$.
	More precisely, for any particle $u$,
	let $d_u$ be the death time of the particle $u$ and
	we use $\tau_D(u)$ to denote the first exit time of the path of $u$ from $D$ defined as
    $$
     \tau_D(u):=\inf\{ s\in [0,d_u), (s, X_u(s))\notin D\}
    $$
    with the convention that $\inf \emptyset=\infty$.
    Put
    $$
    \tau_D:=\inf\{ s\geq 0, (s, B_s)\notin D\}
    $$

    Define the stopping line
	\begin{equation}\label{def_stopping_line}
		L_D := \left\{ (u,t) \in (\cup_{t\geq 0} N_t) \times [0,\infty): u\in N_t, \tau_D(u) = t \right\},
	\end{equation}
    and let $N_D$ be the set of particles stopped on the line, i.e.,
    \begin{equation}\label{def_stopping_set}
        N_D := \{u\in \cup_{t\geq 0} N_t: (u,t)\in L_D \mbox{ for some } t \}.
    \end{equation}
    We refer to \cite{Chauvin91} for the precise definition.

    \begin{lemma}[Many-to-one formula]\label{lemma:many-to-one}
    $\mathrm{(i)}$
    For any nonnegative bounded measurable function $F:[0,\infty)\times \R\rightarrow\R$, it holds that
		\begin{equation}\label{many-to-one}
               \E \left[\sum_{u\in N_{D}} F\left(\tau_D(u), X_u(\tau_D(u))\right)\mathbf{1}_{ (\tau_D(u)<\infty) }\right]  = \bP \left[ e^{\tau_D} F\left(\tau_D, B_{\tau_D}\right), \tau_D<\infty \right].
		\end{equation}
	$\mathrm{(ii)}$	
		For any $t>0$ and any nonnegative bounded measurable function $F:C[0,t]\rightarrow\R$, it holds that
		\begin{equation}\label{many-to-one'}
		    \E \left[\sum_{u\in N_t} F\left(X_u(s),s\leq t\right) \right]  = \bP \left[ e^t F\left(B_s,s\leq t\right) \right].
		\end{equation}
    \end{lemma}

	The proof of Lemma \ref{lemma:many-to-one} (i) will be given in Appendix \ref{sec:appendix_manytoone}. Lemma \ref{lemma:many-to-one} (ii) is from Theorem 8.5 of \cite{HH09}.
 \bigskip

	Now we introduce some results on the additive and derivative martingales of BBM. For $t\geq 0$, define
	\begin{equation}
		W_t := \sum_{u\in N_t} e^{\sqrt{2}(X_u(t) - \sqrt{2}t)}.
	\end{equation}
	Then $\{W_t, t\geq 0;  \P\}$ is called the additive martingale of BBM, which corresponds to $\{W_t(\lambda), t\geq 0\}$ with $\lambda = \underline{\lambda}=\sqrt{2}$ in \cite{Kyprianou04}.
	Similarly, the derivative martingale $\{Z_t, t\geq 0; \P\}$ defined by \eqref{def_derivative} corresponds to $\{\partial W_t(\underline{\lambda})\}$ in \cite{Kyprianou04}.
	By \cite[Theorems 1 and 3]{Kyprianou04}, we have the following result.
	\begin{lemma}\label{lemma:add_deri}
		 For any $y\in\R$, the limits $W_{\infty} := \lim_{t\rightarrow\infty} W_t$ and $Z_{\infty} := \lim_{t\rightarrow\infty} Z_t$ exist $\P$-almost surely. Furthermore, $W_{\infty} = 0$ and $Z_{\infty}\in (0,\infty)$ $\P$-almost surely.
	\end{lemma}
	Note that $Z_t$ can be negative. For any $z>0$, define
	\begin{equation}\label{def_mart_V}
		V_t^z := \sum_{u\in N_t}  (z+\sqrt{2}t-X_u(t)) e^{\sqrt{2}(X_u(t)-\sqrt{2}t) } \1_{\left\{ \inf_{r\leq t} (z+\sqrt{2}r-X_u(r)) > 0 \right\} },\quad t\geq 0.
	\end{equation}
    By \cite[Theorem 9]{Kyprianou04}, $\{V_t^z, t\geq 0\}$ is a (non-negative) martingale. According to \cite[Theorem 13]{Kyprianou04}, for any $z>0$, the limit $V_{\infty}^z := \lim\limits_{t\rightarrow\infty} V_t^z$ exists $\P$-almost surely and
    the convergence also holds in $L^1(\P)$.
	Define
	\begin{equation}\label{def_gamma}
		\gamma^{(z,\sqrt{2})} := \{\omega\in\Omega: \forall t\geq 0, \forall u\in N_t, X_u(t) \leq z + \sqrt{2}t \}.
	\end{equation}
	By the proof of \cite[Corollary 10]{Kyprianou04}, we know that
	\begin{equation}\label{V=Z}
	\lim\limits_{t\rightarrow\infty} V_t^z=	V_{\infty}^z = Z_{\infty}\quad \mbox{ on } \gamma^{(z,\sqrt{2})},
	\end{equation}
	and
	\begin{equation}\label{eq_gamma_to_1}
		\P(\gamma^{(z,\sqrt{2})}) \uparrow 1 \mbox{ as } z\uparrow\infty.
	\end{equation}
    Therefore, $\{Z_t, t\geq 0\}$ can be approximated by $\{V_t^z, t\geq 0\}$ almost surely as $z\to\infty$.

	Since $\{V_t^z, \P\}$ is a non-negative martingale, we can define a probability measure $\P^{(z,\sqrt{2})}$ by
	\begin{align}
		\frac{\d \P^{(z,\sqrt{2})}}{\d \P} \bigg{|}_{\F_t} = \frac{V_t^z}{V_0^z},\quad t>0.
	\end{align}
    It was proved in \cite[Section 6]{Kyprianou04} that, under $ \P^{(z,\sqrt{2})}$, the following spine decomposition is valid:

	\begin{enumerate}[(i)]
		\item There is a distinguished and randomized spine. The diffusion along the spine begins from the origin at time $0$ and $\{z+\sqrt{2}t-X_{\xi}(t), t\geq 0\}$ is a Bessel-3 process started from $z$, where $\xi=\{\xi_t:t\ge0\}$ is the spine, and $\xi_t$ is the spine particle at time $t$ and $X_\xi(t)$ is the shorthand for $X_{\xi_t}(t)$.
		\item The branching rate along the spine is $2$.
		\item The distribution of offspring numbers at each fission point on the spine is the size-biased distribution $\{\tilde{p}_k = kp_k/2, k\geq 1\}$.
		\item The spine is chosen uniformly from all the offsping at the fission time on the spine.
		\item The remaining particles evolve as classical branching Brownian motions.
	\end{enumerate}
	Moreover,
	\begin{equation}\label{eq:spine_decomposition}
		\P^{(z,\sqrt{2})}\left( \xi_t=u \mid \F_t \right) = \frac{(z+\sqrt{2}t-X_u(t)) e^{\sqrt{2}(X_u(t)-\sqrt{2}t) } \1_{\left\{ \inf_{r\leq t} (z+\sqrt{2}r-X_u(r)) > 0 \right\} }}{V_t^z},
	\end{equation}
    see, for example, \cite[Lemma 52]{Berestycki14a}.

   \bigskip
	
	To prove Theorem \ref{thrm1}, we also need some results on  $M_t$ and the extremal process of
	the BBM.
	The following estimate of the tail probability of $M_t$ can be found in \cite[Corollary 10]{ABK12}.
	\begin{lemma}\label{lemma:M_tail}
		There exists $t_0>0$ such that for any $y>1$ and $t\ge t_0$,
		\begin{equation}
			\P\left( M_t \geq \sqrt{2}t - \frac{3}{2\sqrt{2}}\log t + y \right) \leq b y e^{-\sqrt{2}y - \frac{y^2}{2t} + \frac{3}{2\sqrt{2}} y \frac{\log t}{t}}
		\end{equation}
		for some constant $b>0$.
	\end{lemma}
	It follows from Lemma \ref{lemma:M_tail} that
	\begin{equation}\label{eq:M_tail}
		\P\left( M_t \geq \sqrt{2}t - \frac{3}{2\sqrt{2}}\log t + y \right) \leq b y e^{-\sqrt{2}y +1}
	\end{equation}
	for $y>1$ and $t$ sufficiently large, since $ - \frac{y^2}{2t} + \frac{3}{2\sqrt{2}} y \frac{\log t}{t} \leq 1$ in this case.

	\section{Proof of Theorem \ref{thrm1}}\label{Sec3}
	In this section, we fix $\rho<\sqrt{2}$.
	First, we need the following technical lemma, whose proof is postponed to Appendix \ref{sec:appendix}.
	\begin{lemma}\label{lemma:technical}
		There exists $C=C(\rho)>0$ such that for any $x, y>0$,
		\begin{align}\label{eq:estimate_absorb_tail}
			&\limsup_{t\rightarrow\infty} \P (\exists u\in N_t: x+X_u(t) \geq m_t+y, \inf_{r\in [0,t]} (x+X_u(r)-\rho r)<0 ) \\
			&\leq C (1+x+y) e^{-\sqrt{2}y+(2\rho-\sqrt{2})x}.
		\end{align}
	\end{lemma}

	As a direct corollary of Lemma \ref{lemma:technical}, taking $y = x-A$ in \eqref{eq:estimate_absorb_tail} yields the following decay in $x$,
	\begin{align}
		&\limsup_{t\rightarrow\infty} \P (\exists u\in N_t: x+X_u(t) \geq m_t+x-A, \inf_{r\in [0,t]} (x+X_u(r)-\rho r)<0 ) \\
		&\leq C (1+x) e^{-2(\sqrt{2}-\rho)x}.
	\end{align}
    Recall the definition of $\gamma_u(s)$ given by \eqref{def_gmma_u}.
	Based on the lemma above, we obtain the following result, which plays a key role in the proof of the main theorem.
	\begin{lemma}\label{lemma:gamma}
		For any $A>0$, it holds that
		\begin{equation}
			\lim_{s\rightarrow\infty} \limsup_{t\rightarrow\infty} \P\left(\exists u\in N_t: X_u(t) \geq m_t -A, \gamma_u(t) \neq \gamma_u(s) \right) = 0.
		\end{equation}
	\end{lemma}
	\begin{proof}
		Since $\rho < \sqrt{2}$, we can choose $\delta>0$ satisfying $\rho+\delta<\sqrt{2}$.
		For any $t>s>0$, let
		\begin{align}
			I(s,t) &:= \P\left(\exists u\in N_t: X_u(t) \geq m_t -A, X_u(s) \leq (\rho+\delta) s\right),\\
			II(s,t) &:= \P\left(\exists u\in N_t: X_u(t) \geq m_t -A, X_u(s) > (\rho+\delta) s,
            \inf_{r\in [s,t]} ( X_u(r)
            -\rho r) < 0\right).
		\end{align}
		Then
		\begin{align}
			&\P(\exists u\in N_t: X_u(t) \geq m_t -A, \gamma_u(t) \neq \gamma_u(s) )\\
			&\leq  \P\left(\exists u\in N_t: X_u(t) \geq m_t -A,
            \inf_{r\in [s,t]} (X_u(r)-\rho r) < 0 \right)\\
			&\leq I(s,t) + II(s,t).
		\end{align}
		
		Notice that
		\begin{align}
			I(s,t) = &\, \P(\exists u\in N_t: X_u(t) \geq m_t -A, X_u(s) \leq (\rho+\delta)s) \\
			\leq &\, \E\left[ \sum_{v\in N_s} \1_{\left\{ X_v(s)\leq (\rho+\delta)s, \exists  u\succ v, u\in N_t, X_u(t) \geq m_t -A \right\}} \right],
		\end{align}
            where $u\succ v$ means that $u$ is a descendant of $v$.
		It follows from the Markov property that
		\begin{align}	
            I(s,t)\leq &\,
            \E\left[ \sum_{v\in N_s} \1_{\left\{ X_v(s)\leq (\rho+\delta)s \right\}}
            \E\left[  \1_{\left\{ \exists u\succ v, u\in N_t, X_u(t) \geq m_t -A \right\}} \mid \F_s \right]  \right]\\
			=&\, \E\left[ \sum_{v\in N_s} \1_{\left\{ X_v(s)\leq (\rho+\delta)s \right\}} \P(x+M_{t-s}(v) \geq m_t-A) \mid_{x=X_v(s)} \right],
		\end{align}
		where given $\F_s$, $\{M_{t-s}(v):v\in N_s\}$ are independent copies of $M_{t-s}$. Therefore, by the many-to-one formula \eqref{many-to-one'}, we have
		\begin{align}
			I&(s,t) \leq e^s \bE\left[ \1_{\{ B_s\leq (\rho+\delta)s\}} \P(x+M_{t-s}(v) \geq m_t-A)\mid_{x=B_s} \right]\\
			&= e^s \int_{-\infty}^{(\rho+\delta)s}  \frac{1}{\sqrt{2\pi s}} e^{-\frac{x^2}{2s}} \P(x+M_{t-s}(v) \geq m_t-A) \d x \\
			&= e^s \int_{-\infty}^{(\rho+\delta)s}  \frac{1}{\sqrt{2\pi s}} e^{-\frac{x^2}{2s}} \P\left(M_{t-s}(v) \geq m_{t-s}-A-x+\sqrt{2}s-\frac{3}{2\sqrt{2}}\log \frac{t}{t-s}\right) \d x.
		\end{align}
		Let $y(x) = \sqrt{2}s-x-A-\frac{3}{2\sqrt{2}}\log \frac{t}{t-s}$. For $x\in (-\infty,(\rho+\delta)s)$, noticing that $(\rho+\delta)s<\sqrt{2}s$, we see that $y(x)>1$  for $s,t$ large enough. By \eqref{eq:M_tail}, we get that for $s, t$ large enough,
        $$
        I(s,t) \leq e^s \int_{-\infty}^{(\rho+\delta)s}  \frac{1}{\sqrt{2\pi s}} e^{-\frac{x^2}{2s}} by(x) e^{-\sqrt{2}y(x)+1}\d x.
        $$
        Thus there exists a constant $C$ depending only $A$ such that
        $$I(s,t)\leq C \int_{-\infty}^{(\rho+\delta)s}  \frac{1}{\sqrt{2\pi s}} (\sqrt{2}s-x) e^{-\frac{(x-\sqrt{2} s)^2}{2s}}  \d x.$$
        Using the change of variables $z = x - \sqrt{2}s$, we get
        $$I(s,t)\leq C \int_{-\infty}^{-(\sqrt{2}-\rho-\delta) s}  \frac{-z}{\sqrt{2\pi}} e^{-\frac{z^2}{2s}} \d z  = C \frac{s}{\sqrt{2\pi}} e^{-\frac{(\sqrt{2}-\rho-\delta)^2}{2}s}.
        $$
	    Hence
		\begin{equation}
			\lim_{s\rightarrow\infty} \limsup_{t\rightarrow\infty} I(s,t) = 0.
		\end{equation}
		
		Now, we estimate $II(s,t)$. By the Markov property and \eqref{many-to-one'}, we have
		\begin{align}
			&II(s,t) = \P\left(\exists u\in N_t: X_u(t) \geq m_t -A, X_u(s) > (\rho+\delta)s, \inf_{r\in [s,t]} (X_u(r)-\rho r) < 0\right)\\
			&\leq \E\left[ \sum_{v\in N_s} \1_{\left\{ X_v(s) > (\rho+\delta)s; \,\exists u\succ v, u\in N_t, X_u(t) \geq m_t -A,\, \inf_{r\in [s,t]} (X_u(r)-\rho r)<0  \right\}} \right]\\
			&= \E\left[ \sum_{v\in N_s} \1_{\left\{ X_v(s) > (\rho+\delta)s \right\}}\E\left[  \1_{\left\{ \exists u\succ v, u\in N_t: X_u(t) \geq m_t -A, \,\inf_{r\in [s,t]} (X_u(r)-\rho r)<0  \right\}} \mid \F_s \right]  \right]\\
			&= e^s \int_{(\rho+\delta)s}^{\infty} \frac{1}{\sqrt{2\pi s}} e^{-\frac{x^2}{2s}} \E\left( \1_{\left\{\exists u\in N_{t-s}: x+X_u(t-s) \geq m_t -A, \inf\limits_{r\in [0,t-s]} (x+X_u(r)-\rho (r+s))<0 \right\}} \right) \d x\\
			&= e^s \int_{(\rho+\delta)s}^{\infty} \frac{1}{\sqrt{2\pi s}} e^{-\frac{x^2}{2s}} \E\left( \1_{\left\{\exists u\in N_{t-s}: x-\rho s+X_u(t-s) \geq m_t -\rho s-A, \inf\limits_{r\in [0,t-s]} (x-\rho s+X_u(r)-\rho r)<0 \right\}} \right) \d x.
		\end{align}
		Note that
		\begin{align}
			&\left\{x-\rho s+X_u(t-s) \geq m_t -\rho s-A \right\}\\
			=& \left\{x-\rho s+X_u(t-s) \geq m_{t-s} + (\sqrt{2}-\rho)s-A - \frac{3}{2\sqrt{2}}\log\frac{t}{t-s} \right\}.
		\end{align}
        Using Lemma \ref{lemma:technical} with $t$ replaced by $t-s$, $x$ replaced by $x-\rho s$ and $y$ replaced by $(\sqrt{2}-\rho)s-A-\frac{3}{2\sqrt{2}}\log\frac{t}{t-s}$, we have that there exists a constant $C$ depending only on $\rho$ and $A$, such that
		\begin{align}
			\limsup_{t\rightarrow\infty} II(s,t) &\leq  C e^s \int_{(\rho+\delta)s}^{\infty} \frac{1}{\sqrt{2\pi s}} e^{-\frac{x^2}{2s}} e^{-\sqrt{2}\left( (\sqrt{2}-\rho)s-A   \right) } \\
			 &\hspace{3em}\left([(x-\rho s)+(\sqrt{2}-\rho)s ]e^{(2\rho-\sqrt{2})(x-\rho s)}\right) \d x.
    \end{align}
    By the change of variables $\tilde x=x-\rho s$, we get
    \begin{align}
		\limsup_{t\rightarrow\infty} II(s,t)
	    &\leq C e^{\sqrt{2}A}\int_{\delta s}^{\infty} \frac{1}{\sqrt{2\pi s}} e^{-\frac{(\tilde x+\rho s)^2}{2s}} e^{-s+\sqrt{2}\rho s}
	    \left(\tilde x+(\sqrt{2}-\rho)s\right)
	    e^{(2\rho-\sqrt{2})\tilde x} \d \tilde x\\
		&= Ce^{\sqrt{2}A} \int_{\delta s}^{\infty} \frac{x+(\sqrt{2}-\rho)s}{\sqrt{2\pi s}} e^{-\frac{(x+(\sqrt{2}-\rho)s)^2}{2s}} \d x\\
		&= Ce^{\sqrt{2}A} \int_{\delta \sqrt{s}}^{\infty} \frac{ \sqrt{s}z+(\sqrt{2}-\rho)s}{\sqrt{2\pi }} e^{-\frac{(z+(\sqrt{2}-\rho)\sqrt{s})^2}{2}} \d z \rightarrow 0, \mbox{ as } s\rightarrow\infty,
	\end{align}
    where we used the change of variables $z={x}/\sqrt{s}$ in the last equality.
    This completes the proof.
	\end{proof}
		
	In \cite{BH17}, Bovier and Hartung studied
	the convergence of the extremal process of BBM by adding an extra dimension that encodes the ``location'' of the particle in the underlying Galton-Watson tree. We adapt some ideas from \cite{BH17} to prove Theorem \ref{thrm1}.

	We prove the convergence of $\sum_{u\in N_t} \delta_{(X_u(t)-m_t,-\inf\limits_{s\leq t} \{X_u(s)-\rho s\} )}$ in two main steps.
	In the first step, we show that the points of the local extrema converge to the desired Poisson point process, which is formalized in the proposition below. The second step is the proof of Theorem \ref{thrm1}, in which we incorporate the siblings of local extremal particle and shows that all of them converge to the decorated Poisson point process.
	
    We also use the concept of thinning classes introduced in \cite{ABK12, BH17}. More precisely, we consider the following collection of particles define below.
	For $\q \in (s,t)$, define
	\begin{align}\label{def_N^rd}
		N_t^{\q } := \left\{ u\in N_t: \exists \, v\in N_{\q } \mbox{ s.t. }
 u\succ v
 \mbox{ and } X_u(t) = \max_{w\in N_t,w\succ v} X_w(t)  \right\}.
	\end{align}
    Define the corresponding point process by
    \begin{equation}
    	\e_t^{\q } := \sum_{u\in N_t^{\q }} \delta_{\left( X_u(t)-m_t, -\gamma_u(t) \right) }, t>0,
    \end{equation}
    which is a thinning of $\e_t$. For the  thinning measures $\{\e_t^{\q }, t>0\}$, we have the following limit result.  We will use $\mathrm{PPP}(\mu)$ to denote a Poisson point process with intensity $\mu$.
	\begin{prop}\label{prop:q-thinning}
		Under $\P$, it holds that
		\begin{align}
			\lim_{\q \rightarrow\infty} \lim_{t\rightarrow\infty} \e_t^{\q }= \mathrm{PPP}(\sqrt{2}C_*e^{-\sqrt{2}y}\d y \times \d Z^{-\rho}) \mbox{ in law, }
		\end{align}
		where $Z^{-\rho}$ is the random measure on $\R_+$ characterized in Proposition \ref{prop_truncated_derivative}. Moreover,
		\begin{equation}\label{eq:converge_q_thinnig}
			\lim_{s\rightarrow\infty} \lim_{\q \rightarrow\infty} \sum_{v\in N_{\q }} \delta_{\left( X_v(\q )-\sqrt{2}\q +M(v), -\gamma_v(s) \right) } = \mathrm{PPP}(\sqrt{2}C_*e^{-\sqrt{2}y}\d y \times \d Z^{-\rho}) \mbox{ in law, }
		\end{equation}
		where $\{ M(v): v\in N_{\q }\}$
		are i.i.d. with law $w$ given by \eqref{eq:Mt_speed}.
	\end{prop}

	The proof of Proposition \ref{prop:q-thinning} relies on the definition of $Z^{-\rho}(\d x)$, which needs Proposition \ref{prop_truncated_derivative}. We first prove Proposition \ref{prop_truncated_derivative}, then Proposition \ref{prop:q-thinning}. Then we prove Theorem \ref{thrm1}, and finally Proposition \ref{prop:proportion}.

	\begin{proof}[Proof of Proposition \ref{prop_truncated_derivative}]
    Recall that, for fixed $x>0$ and $s>0$,  $Z^{-\rho}(x,s,t)$ is a martingale for $t>s$.
    For fixed $s,x,z\in (0,\infty)$, define
    \begin{align}
    	V^z(x,s,t) := \sum_{u\in N_t} (\sqrt{2}t-X_u(t)+z) e^{\sqrt{2}(X_u(t)-\sqrt{2}t)} \1_{\left\{ \inf\limits_{r\leq s} (X_u(r)-\rho r) > -x ,\,  \inf\limits_{r\leq t} (\sqrt{2}r - X_u(r) + z) \geq 0 \right\} }.
    \end{align}
    Recall that $\{V_t^z\}$, defined by \eqref{def_mart_V},  is a non-negative martingale. Since $\1_{\left\{ \inf_{r\leq s} (X_u(r)-\rho r) > -x \right\}}$ does not depend on $t$, we know that $V^z(x,s,t)$ is a non-negative martingale for $t>s$. Therefore,
    \begin{align}
    	V^z(x,s) := \lim_{t\rightarrow\infty} V^z(x,s,t) \mbox{ exists $\P$-a.s.}
    \end{align}
    Recall that $\gamma^{(z,\sqrt{2})}$ is defined by \eqref{def_gamma}. Notice that on $\gamma^{(z,\sqrt{2})}$,
    \begin{equation}
    	V^z(x,s,t) = Z^{-\rho}(x,s,t) + z \sum_{u\in N_t}  e^{\sqrt{2}(X_u(t)-\sqrt{2}t)} \1_{\left\{ \inf_{r\leq s} (X_u(r)-\rho r) > -x \right\} }
    \end{equation}
    and
    \begin{equation}
    	\lim_{t\rightarrow\infty} z \sum_{u\in N_t}  e^{\sqrt{2}(X_u(t)-\sqrt{2}t)} \1_{\left\{ \inf_{r\leq s} (X_u(r)-\rho r) > -x \right\} } \leq z \lim_{t\rightarrow\infty} W_t = 0.
    \end{equation}
    Therefore, on $\gamma^{(z,\sqrt{2})}$,  we have
    $Z^{-\rho}(x,s) := \lim_{t\rightarrow\infty} Z^{-\rho}(x,s,t)$ exists and $Z^{-\rho}(x,s) = V^z(x,s) \geq 0$. Moreover, it follows from \eqref{eq_gamma_to_1} that $Z^{-\rho}(x,s) := \lim_{t\rightarrow\infty} Z^{-\rho}(x,s,t)$ exists $\P$-almost surely and $Z^{-\rho}(x,s)\geq 0$.

     Since, for fixed $x>0$, $\1_{\{ \inf_{r\leq s} (X_u(r)-\rho r) > -x \} }$ is non-increasing in $s$,
    we have that $V^z(x,s)$ is non-increasing in $s$ and so is $Z^{-\rho}(x,s)$. It follows from $0\leq V^z(x,s) \leq V^z_{\infty}$ that $0\leq Z^{-\rho}(x,s) \leq Z_{\infty}$. Hence, $Z^{-\rho}(x) := \lim_{s\rightarrow\infty} Z^{-\rho}(x,s)$ exists $\P$-almost surely. Note that $V^z(x,s)$ is non-decreasing in $x$ $\P$-almost surely, and so is $Z^{-\rho}(x,s)$. Therefore, $Z^{-\rho}(x)$ is non-decreasing with respect to $x$.
    \end{proof}

    \begin{proof}[Proof of Proposition \ref{prop:q-thinning}]
    	Let $f\in C_c^+(\R\times \R_+)$, where $C_c^+(\R\times \R_+)$ denotes the class of all the non-negative continuous function on $C_c(\R\times \R_+)$ with compact support. We need to study the limit of $\E\left( e^{ -\sum_{u\in N_t^{\q }} f(X_u(t)-m_t, -\gamma_u(t)) }  \right)$.
        By uniform continuity, every $f\in C_c^+(\mathbb R\times\R_+)$ admits non-negative rectangular simple functions $f_n^-$ and $f_n^+$, all supported in a common compact set, such that
        \begin{equation}
        f_n^-\le f\le f_n^+,
        \qquad
        \|f_n^\pm-f\|_\infty\rightarrow0.
        \end{equation}
        For every locally finite point measure $\nu$,
        \[
        e^{-\langle f_n^+,\nu\rangle}
        \le e^{-\langle f,\nu\rangle}
        \le e^{-\langle f_n^-,\nu\rangle}.
        \]
        Moreover, by bounded converge theorem, the Laplace functional of the proposed limiting point process evaluated at $f_n^\pm$ converges to its value at $f$. Thus, applying the squeeze argument to the corresponding upper and lower limits, it suffices to establish the claimed Laplace functional limits for simple functions of the form
        \begin{equation}
        f(y,z)=\sum_{i=1}^{n}a_i \1_{A_i\times B_i}(y,z),
        \end{equation}
        where $a_i\ge0$,
        \begin{equation}
        A_i=(\underline A_i,\overline A_i],
        \qquad
        B_i=(\underline B_i,\overline B_i],
        \end{equation}
        and the bounded rectangles $A_i\times B_i$ are pairwise disjoint.
    	Let $-A = \min_{1\leq i\leq n} \underline{A}_i$ and
    	\begin{equation}
    		\mathcal{A}_{s,t}(A) := \left\{ \forall u\in N_t \mbox{ with } X_u(t)-m_t \geq -A: \gamma_u(t) = \gamma_u(s) \right\}.
    	\end{equation}
    	It follows from Lemma \ref{lemma:gamma} that
    	\begin{equation}\label{eq:estimate_mathcal_A}
    		\lim_{s\rightarrow\infty} \limsup_{t\rightarrow\infty} \P\left( (\mathcal{A}_{s,t}(A))^c \right) = 0.
    	\end{equation}
    	Hence, it suffices to show that
    	\begin{equation}\label{eq:laplace_on_A}
    		\E\left( e^{ -\sum_{u\in N_t^{\q }} f(X_u(t)-m_t, -\gamma_u(t)) } \1_{\mathcal{A}_{s,t}(A)} \right)
    	\end{equation}
    	converges to the desired limit. Note that on $\mathcal{A}_{s,t}(A)$,
    	\begin{equation}
    		\sum_{u\in N_t^{\q }} f(X_u(t)-m_t, -\gamma_u(t)) = \sum_{u\in N_t^{\q }} f(X_u(t)-m_t, -\gamma_u(s)).
    	\end{equation}
        Using \eqref{eq:estimate_mathcal_A} again, we get that \eqref{eq:laplace_on_A} is equal to
        \begin{equation}\label{eq:laplace_s}
            \E\left( e^{ -\sum_{u\in N_t^{\q }} f(X_u(t)-m_t, -\gamma_u(s)) } \right) + o_{s,t}(1),
        \end{equation}
        where $o_{s,t}(1)$ is a quantity such that $\lim_{s\rightarrow\infty} \lim_{t\rightarrow\infty} o_{s,t}(1) = 0$.
        Conditioned on $\F_{\q }$, we have
        \begin{align}\label{eq:laplace_conditional}
        	&\lim_{t\rightarrow\infty} \E\left( e^{ -\sum_{u\in N_t^{\q }} f(X_u(t)-m_t, -\gamma_u(s)) } \right)\\
        	&= \lim_{t\rightarrow\infty} \E\left[ \prod_{v\in N_{\q }} \E\left[ e^{ -f\left(X_v(\q )-m_t+m_{t-\q } + \max_{u\in N_t, u\succ v} (X_u(t)-X_v(\q )) - m_{t-\q }, -\gamma_v(s) \right) } \mid \F_{\q } \right] \right]\\
        	&= \E\left[ \prod_{v\in N_{\q }} \E\left[ e^{ -f\left(X_v(\q )-\sqrt{2}\q  + M, -\gamma_v(s) \right) } \mid \F_{\q } \right] \right],
        \end{align}
        where we used the fact that $\lim_{t\rightarrow\infty} \frac{3}{2\sqrt{2}}\log\frac{t}{t-\q } = 0$, and given $\F_{\q }$, $\max_{u\in N_t, u\succ v} (X_u(t)-X_v(\q )) - m_{t-\q }$ converges weakly, as $t\to\infty$,
        to the limit of
        $M_t-m_t$, see \eqref{eq:Mt_speed}.
        The equation above is also the Laplace functional of $\sum_{v\in N_{\q }} \delta_{( X_v(\q )-\sqrt{2}\q +M(v), -\gamma_v(s)) }$. Define $S_v(\q ) := \sqrt{2}\q -X_v(\q )$. Set
        \begin{align}
        	F(-S_v(\q ), -\gamma_v(s)) := 1 - \E\left[ e^{ -f\left(-S_v(\q ) + M, -\gamma_v(s) \right) } \mid \F_{\q } \right].
        \end{align}
        Then, by \eqref{eq:laplace_conditional},
        \begin{equation}\label{Laplacelimit}
             \lim_{t\rightarrow\infty} \E\left( e^{ -\sum_{u\in N_t^{\q }} f(X_u(t)-m_t, -\gamma_u(s)) } \right)=\E \left[ e^{ \sum_{v\in N_{\q }} \log\left[1-F(-S_v(\q ), -\gamma_v(s))\right] } \right].
        \end{equation}

        Note that $\min_{v\in N_{\q }} S_v(\q ) \rightarrow\infty$ almost surely as $\q \rightarrow\infty$. This implies that
        \begin{equation}
        	\max_{v\in N_{\q }} F(-S_v(\q ), -\gamma_v(s)) \rightarrow 0 \mbox{ as } \q \rightarrow\infty.
        \end{equation}
        It follows from $-x-x^2<\log (1-x) < -x$ for $0<x<1/2$ that
        \begin{align}
        	& \E \left[ e^{ -\sum_{v\in N_{\q }} \left[F(-S_v(\q ), -\gamma_v(s)) +F(-S_v(\q ), -\gamma_v(s))^2\right] } \right]\\
        	 &\leq \E \left[ e^{ \sum_{v\in N_{\q }} \log\left[1-F(-S_v(q), -\gamma_v(s))\right] } \right] \label{eq:laplace_upper_lower} \\
        	&\leq \E \left[ e^{ -\sum_{v\in N_{\q }} F(-S_v(\q ), -\gamma_v(s))} \right].
        \end{align}
        Now we prove the lower bound and upper bound have the same limit. Since $f$ is a simple function, we have
        \begin{equation}
        	F(-S_v(\q ), -\gamma_v(s)) = \sum_{i=1}^n (1-e^{-a_i}) \1_{B_i}(-\gamma_v(s)) \int_{A_i+S_v(\q )} \d w.
        \end{equation}
        Using the asymptotic behaviors \eqref{eq:w_asymptotic}, we get that
        \begin{align}
        	F(-S_v(\q ), -\gamma_v(s)) \sim \sum_{i=1}^n (1-e^{-a_i}) \1_{B_i}(-\gamma_v(s)) C_* &\left[(\underline{A}_i+S_v(\q )) e^{-\sqrt{2}(\underline{A}_i+S_v(\q )) } \right.\\
        	&\, \left.-  (\overline{A}_i+S_v(\q )) e^{-\sqrt{2}(\overline{A}_i+S_v(\q )) }\right],
        \end{align}
        where $\sim$ means that the ratio of the left-hand and right-hand sides converges to $1$, as $\q \rightarrow\infty$, $\P$-almost surely. Since the additive martingale $W_{\q } = \sum_{v\in N_{\q }} e^{-\sqrt{2}S_v(\q )}$ converges to $0$ $\P$-almost surely, we know that
        \begin{align}
        	&\lim_{\q \rightarrow\infty} \sum_{v\in N_{\q }} F(-S_v(\q ), -\gamma_v(s))\\
        	&= \lim_{\q \rightarrow\infty} C_* \sum_{v\in N_q} \sum_{i=1}^n (1-e^{-a_i}) \1_{B_i}(-\gamma_v(s)) \left[S_v(\q ) e^{-\sqrt{2}(\underline{A}_i+S_v(\q )) } -  S_v(\q ) e^{-\sqrt{2}(\overline{A}_i+S_v(\q )) }\right]\\
        	&= C_*  \sum_{i=1}^n (1-e^{-a_i}) (Z^{-\rho}(\overline{B}_i,s)-Z^{-\rho}(\underline{B}_i,s)) \left(e^{-\sqrt{2}\underline{A}_i } -  e^{-\sqrt{2}\overline{A}_i }\right).
        \end{align}
        This yields the limit of the upper bound in \eqref{eq:laplace_upper_lower}.
        For the lower bound in \eqref{eq:laplace_upper_lower}, we have
        \begin{align}
        	\sum_{v\in N_{\q }} F(-S_v(\q ), -\gamma_v(s))^2 = O(\sum_{v\in N_{\q }}S_v(\q )^2 e^{-2\sqrt{2}S_v(\q ) }) \rightarrow 0, \quad \mbox{as } \q \rightarrow\infty.
        \end{align}
        Therefore,  the limits of the upper bound and low bound are the same. And then
        \begin{align}\label{q-infty}
        	&\lim_{\q \rightarrow\infty} \E \left[ e^{ \sum_{v\in N_{\q }} \log\left[1-F(-S_v(\q), -\gamma_v(s))\right] } \right]\\
        	&= \E\left[ \exp\left\{ -C_* \sum_{i=1}^n (1-e^{-a_i}) (Z^{-\rho}(\overline{B}_i,s) - Z^{-\rho}(\underline{B}_i,s)) \left(e^{-\sqrt{2}\underline{A}_i } -  e^{-\sqrt{2}\overline{A}_i }\right) \right\} \right].
        \end{align}
               Since \eqref{eq:laplace_conditional} holds for any $\q \in (s,t)$,         using \eqref{Laplacelimit} and \eqref{q-infty},
               we have
        \begin{align}
        	&
        	\lim_{\q \rightarrow\infty}
        	\lim_{t\rightarrow\infty} \E\left( e^{ -\sum_{u\in N_t^{\q }} f(X_u(t)-m_t, -\gamma_u(t)) }  \right)\\
        	&= \lim_{s\rightarrow\infty}
        	\lim_{\q \rightarrow\infty}
        	\lim_{t\rightarrow\infty} \E\left( e^{ -\sum_{u\in N_t^{\q }} f(X_u(t)-m_t, -\gamma_u(s)) } \right)\\
        	&= \lim_{s\rightarrow\infty} \lim_{\q \rightarrow\infty} \E \left[ e^{ \sum_{v\in N_{\q }} \log\left[1-F(-S_v(\q ), -\gamma_v(s))\right] } \right]\\
        	&= \lim_{s\rightarrow\infty} \E\left[ \exp\left\{ -C_* \sum_{i=1}^n (1-e^{-a_i}) (Z^{-\rho}(\overline{B}_i,s) - Z^{-\rho}(\underline{B}_i,s)) \left(e^{-\sqrt{2}\underline{A}_i } -  e^{-\sqrt{2}\overline{A}_i }\right) \right\} \right]\\
        	&= \E\left[ \exp\left\{ \int (e^{-f(y,z)}-1) Z^{-\rho}(\d z) \sqrt{2}C_*e^{-\sqrt{2}y} \d y  \right\} \right],
        \end{align}
       which is the Laplace functional of $\mathrm{PPP}(\sqrt{2}C_*e^{-\sqrt{2}y}\d y \times \d Z^{-\rho})$. We complete the proof.
    \end{proof}

    \begin{proof}[Proof of Theorem \ref{thrm1}]
    	The proof is similar to that of Theorem 3.1 of \cite{BH17}. For completeness, we give the details here.
    	
    	For any $t>0$ and $u,v\in N_t$, define
    	\begin{equation}
    		d(u,v) := \sup\{s\geq 0: \exists w\in N_s \mbox{ satisfies } w<u \mbox{ and } w<v\},
    	\end{equation}
    	which is the death time of their most recent common ancestor. For $v\in N_{\q }$, define
    	\begin{equation}
    		v_t^{\max}: = \arg\max_{u\in N_t, u\succ v} X_u(t)
    	\end{equation}
    	and
    	\begin{equation}
    		\Delta_{t,r}^{(v)}: = \sum_{u\in N_t, u\succ v} \delta_{X_u(t)-X_{v_t^{\max}}(t)} \1_{\{d(u,v_t^{\max}) > t-r \}}.
    	\end{equation}
    	Then the limit $\Delta_{r}^{(v)} := \lim_{t\rightarrow\infty} \Delta_{t,r}^{(v)}$ exists and it is an independent copy of $\Delta_r$, which is defined by
    	\begin{equation}
            \Delta_r := \lim_{t\rightarrow\infty} \sum_{u\in N_t} \delta_{X_u(t)-M_t} \1_{\{ d(u, \arg\max_{v\in N_t} X_v(t)) > t-r \}}.
    	\end{equation}
    	By \cite[Theorem 2.3]{ABBS13}, we have that
    	\begin{align}\label{eq:converge_max_decoration}
    	    \left( X_v(\q )-\sqrt{2}\q  + \max_{u\in N_t, u\succ v} (X_u(t)-X_v(\q)-m_{t-q}), \Delta_{t,r}^{(v)} \right)_{v\in N_{\q }} \\
    		\overset{d.}{\rightarrow} \left( X_v(\q )-\sqrt{2}\q  + M(v), \Delta_{r}^{(v)} \right)_{v\in N_{\q }},
    	\end{align}
    	as $t\rightarrow\infty$, where $\{M(v):v\in N_{\q }\}$ are independent copies of $M$ with
        distribution function $w$ given by \eqref{travelling}.
        Moreover, $\Delta_{r}^{(v)}$ is independent of $\{M(v):v\in N_{\q }\}$.

    	Using an argument similar to that in the proof of Proposition \ref{prop:q-thinning},
    	we get that for any $f\in C_c^+(\R\times\R_+)$,
    	\begin{align}\label{eq:extreme}
    		\E \left( e^{ -\sum_{u\in N_t} f(X_u(t)-m_t, -\gamma_u(t)) } \right) =  \E \left( e^{ -\sum_{u\in N_t} f(X_u(t)-m_t, -\gamma_u(s)) } \right) + o_{s,t}(1),
    	\end{align}
        where $o_{s,t}(1)$ is a quantity such that $\lim_{s\rightarrow\infty} \lim_{t\rightarrow\infty} o_{s,t}(1) = 0$.
        For $D\subset \R$, define the event
        \begin{align}
        	\mathcal{C}_{t,s}(D) := \left\{ \forall u,v\in N_t \mbox{ with } X_u(t),X_v(t)\in m_t+D: d(u,v) \notin (s,t-s) \right\}.
        \end{align}
        It follows from \cite[Theorem 2.1]{ABK11} that for any compact set $D\subset \R$,
        \begin{equation}\label{eq:common_ancestor}
        	\lim_{s\rightarrow\infty} \sup_{t>3s} \P\left( (\mathcal{C}_{t,s}(D))^c \right) = 0.
        \end{equation}
    	Since $f$ has compact support, there exist compact sets $D_f\subset \R$ and $\widetilde D_f\subset\R_+$ such that $\mathrm{supp} f \subset D_f\times \widetilde D_f$. Put $a= \inf D_f$ and $D_f^k = [a,k]$, then for $k$ sufficiently large, we have $D_f\subset D_f^k$.
        Note that 
    	\begin{equation}\label{decompose-1}
    	    \1 = \1_{(H_t^k)^c} + \1_{H_t^k \cap (\mathcal{C}_{t,s}(D_f^k))^c } + \1_{H_t^k \cap \, (\mathcal{C}_{t,s}(D_f^k)) },
    	\end{equation}
    	where $H_t^k := \{M_t-m_t \leq k \}$. On the event $H_t^k \cap \, (\mathcal{C}_{t,s}(D_f^k))$, any $u\in N_t$ has position $X_u(t)\leq m_t+k$ and any two particles $u,v$ with position $X_u(t),X_v(t)\in m_t+D_f^k$ have their common ancestor either smaller than $s$ or larger than $t-s$.  Recall that for $q\in (s,t)$, $N_t^{\q}$ is defined by \eqref{def_N^rd}, and $N_t = \cup_{u\in N_t^{\q}} \{v\in N_t: d(u,v)\geq q \}$.  Therefore, we have
    	\begin{align}\label{on-H-C}
        	&\,\E \left( e^{ -\sum_{u\in N_t} f(X_u(t)-m_t, -\gamma_u(s)) } \1_{H_t^k \cap \, (\mathcal{C}_{t,s}(D_f^k)) } \right)\\
    		=&\, \E \left(  e^{ -\sum_{u\in N_t^{\q }}  \sum_j f(X_u(t)-m_t + (\Delta_{t,\q }^{(u(\q ))})_j , -\gamma_u(s))} \1_{H_t^k \cap \, (\mathcal{C}_{t,s}(D_f^k)) } \right)
    	\end{align}
    	where $u(\q )$ is the ancestor of $u$ at time $\q $ and $(\Delta_{t,\q }^{(u(\q ))})_j$ are the atoms of $\Delta_{t,\q }^{(u(\q ))}$. For fixed $k>0$, \eqref{eq:common_ancestor} yields that 
        \begin{equation}\label{limit-C-Ac}
            \lim_{s\rightarrow\infty} \sup_{t>3s} \P\left( (\mathcal{C}_{t,s}(D_f^k))^c \right) = 0.
        \end{equation}
        By \eqref{eq:Mt_speed}, we have
    	\begin{equation}\label{limit-Fc}
    		\lim_{k\rightarrow\infty} \lim_{t\rightarrow\infty} \P\left( (H_t^k)^c \right) = \lim_{k\rightarrow\infty} (1-w(k)) = 0.
    	\end{equation}
        Thus, combining  \eqref{decompose-1}, \eqref{on-H-C}, \eqref{limit-C-Ac} and \eqref{limit-Fc}, we have
    	\begin{align}\label{eq:extreme2}
    		&\E \left( e^{ -\sum_{u\in N_t} f(X_u(t)-m_t, -\gamma_u(s)) } \right)\\
    		=& \E \left(  e^{ -\sum_{u\in N_t^{\q }}  \sum_j f(X_u(t)-m_t + (\Delta_{t,\q }^{(u(\q ))})_j , -\gamma_u(s))} \right) + o_{k,s,t}(1)
    	\end{align}
 	    where $o_{k,s,t}(1)$ is a quantity such that $\lim_{k\rightarrow\infty} \lim_{s\rightarrow\infty} \lim_{t\rightarrow\infty} o_{k,s,t}(1) = 0$.
    	Moreover, equation \eqref{eq:converge_max_decoration} yields that
    	\begin{align}
    		&\lim_{t\rightarrow\infty} \sum_{u\in N_t^{\q }}  \sum_j \delta_{(X_u(t)-m_t, -\gamma_u(s)) + ((\Delta_{t,\q }^{(u(\q ))})_j ,0)}\\
    		&=\sum_{v\in N_{\q }}  \sum_j \delta_{(X_v(\q )-\sqrt{2}\q +M(v), -\gamma_v(s)) + ((\Delta_{\q }^{(v)})_j ,0)}.
    	\end{align}
        Letting  $\q \to\infty$, $s\to\infty$ and 
        then $k\to\infty$,
        using \eqref{eq:converge_q_thinnig}, \eqref{eq:extreme} and \eqref{eq:extreme2}, we get the desired result.
        This completes the proof of Theorem \ref{thrm1}.
    \end{proof}

    Now, we prove the properties of the random measure $Z^{-\rho}(\d x)$.

    \begin{proof}[Proof of Proposition \ref{prop:proportion}]
    (i)
    Recall the definitions of $\mu_n$  given by  \eqref{def-mu-n}.
    Notice that
    	\begin{align}
    		Z_{n+s} =&\,  \sum_{u\in N_n} e^{\sqrt{2} X_u(n)-2n} \sum_{v\in N_{n+s},\, v\succ u} (\sqrt{2}s -  (X_{v}(n+s)-X_u(n))) e^{\sqrt{2}(X_{v}(n+s)-X_u(n)) - 2s}\\
    		&\,+ \sum_{u\in N_n} e^{\sqrt{2} X_u(n)-2n} \sum_{v\in N_{n+s},\, v\succ u} X_u(n) e^{\sqrt{2}(X_{v}(n+s)-X_u(n)) - 2s}
    	\end{align}
    	Letting $s$ tend to infinity on the both sides of previous equation, the branching property and $W_\infty=0$ $\P$-almost surely give
    	\begin{equation}\label{eq:consistency}
    		Z_\infty = \sum_{u\in N_n}Z_u(n),\qquad
    		Z_u(n)=\sum_{v\in N_m,\,v\succ u}Z_v(m),\quad m\ge n, \quad \mbox{$\P$-a.s.},
    	\end{equation}
    	where the second identity can be obtained by using a similar argument.
        These identities hold simultaneously    at all integer levels. Moreover, $Z_u(n)>0$ for every such $n,u\in N_n$ almost
    	surely, since $L\ge1$ and the derivative martingale limit is strictly
    	positive on survival.
    	
    	For integers $m\ge n$, define
    	\begin{equation}\label{eq:bad-mass}
    		D_{n,m}=\sum_{v\in N_m}Z_v(m)
    		\1_{\{\inf_{n\le r\le m}(X_v(r)-\rho r)<0\}}.
    	\end{equation}
    	By \eqref{eq:consistency}, $D_{n,m}$ is non-decreasing in $m$. Therefore, we can define
    	\begin{equation}
    		D_n=\lim_{m\to\infty}D_{n,m}.
    	\end{equation}
    	Then, using \eqref{eq:consistency} again, we have that $D_n$ is non-increasing in $n$ and $0\le D_n\le Z_\infty$.
    	We claim that
    	\begin{equation}\label{eq:vanishing}
    		D_n\downarrow0,\quad\P\text{-a.s.}
    	\end{equation}
    	To see this, fix $z>0$ and let
    	\[
    	\gamma_u^{(z,\sqrt{2})}(t)=\left\{\sup_{r\le t}(X_u(r)-\sqrt{2} r)<z\right\}.
    	\]
    	It follows from \eqref{def_gamma} and \eqref{eq_gamma_to_1} that
    	$\P(\bigcup_{z\in N_+}\gamma^{(z,\sqrt{2})}) = 1$. For $t\ge m$, set
    	\begin{align*}
    		V_{n,m}^z(t)
    		&=\sum_{u\in N_t}(z+\sqrt{2} t-X_u(t))e^{\sqrt{2} X_u(t)-2t}\1_{\gamma_u^{(z,\sqrt{2})}(t)}
    		\1_{\{\inf_{n\le r\le m}(X_u(r)-\rho r)<0\}},
    	\end{align*}
    	which is non-negative. Note that on $\gamma^{(z,\sqrt{2})}$, the upper truncation is inactive. Combining this with the branching property and $W_\infty=0$, we have
    	$\lim_{t\rightarrow\infty} V_{n,m}^z(t) =  D_{n,m}$ $\P$-almost surely. Then, by Fatou's lemma and \eqref{eq:spine_decomposition}, we get that
    	\begin{align}
    		&\qquad\E[D_{n,m}\1_{\gamma^{(z,\sqrt{2})}}]
    		\le\liminf_{t\to\infty}\E V_{n,m}^z(t) \label{eq:fatou}\\
    		&= \liminf_{t\rightarrow\infty} z \E\left[ \frac{V_t^z}{z} \sum_{u\in N_t} \frac{(z+\sqrt{2}t-X_u(t)) e^{\sqrt{2}(X_u(t)-\sqrt{2}t)} \1_{\gamma_u^{(z,\sqrt{2})}(t)} }{V_t^z} \1_{\left\{ \inf\limits_{n\le r\le m}(X_u(r)-\rho r)<0\right\}} \right]\\
    		&= \liminf_{t\rightarrow\infty} z\E^{(z,\sqrt{2})}\left[  \sum_{u\in N_t} \P^{(z,\sqrt{2})}(\xi_t=u\mid \F_t) \1_{\left\{ \inf\limits_{n\le r\le m}(X_u(r)-\rho r)<0\right\}} \right]\\
    		&=   z\P^{(z,\sqrt{2})}\left(   \inf\limits_{n\le r\le m}(X_{\xi}(r)-\rho r)<0 \right).
    	\end{align}
        Under $\P^{(z,\sqrt{2})}$, $\{z+\sqrt{2}t-\xi_t: t\geq 0\}$ is a Bessel-3 process starting from $z$ and $\lim_{t\rightarrow\infty} (z+\sqrt{2}t-\xi_t)/t = 0$. Thus,
    	\[
    	X_\xi(t)-\rho t = z+(\sqrt{2}-\rho)t-(z+\sqrt{2}t-X_{\xi}(t))\rightarrow+\infty,
    	\quad\P^{(z,\sqrt{2})}\text{-a.s.},
    	\]
    	Letting $m\to\infty$ in
    	\eqref{eq:fatou}, we obtain
    	\begin{equation}\label{eq:tail-bound}
    		\E[D_n\1_{\gamma^{(z,\sqrt{2})}}]
    		\le z\P^{(z,\sqrt{2})}\bigl(\exists t\ge n:X_\xi(t)-\rho t<0\bigr)
    		\rightarrow0.
    	\end{equation}
    	Since $D_n$ is non-increasing, its limit is zero almost surely on each
    	$\gamma^{(z,\sqrt{2})}$. Taking the union over positive integer $z$ yields
    	\eqref{eq:vanishing}.

        Recall  the definition of $\gamma_u(s)$  given by \eqref{def_gmma_u}.
    	For $v\in N_m$, $\gamma_v(m)<\gamma_v(n)$ implies
    	$\inf_{n\le r\le m}(X_v(r)-\rho r)<0$. Using
    	\eqref{eq:consistency} to express both measures at level $m$, we get
    	\begin{align}\label{eq:TV}
    		\|\mu_m-\mu_n\|_{\mathrm{TV}} &= \left\| \sum_{v\in N_m} Z_v(m) \delta_{-\gamma_v(m)} - \sum_{u\in N_n} \sum_{v\in N_m, \,v\succ u} Z_v(m) \delta_{-\gamma_u(n)} \right\|_{\mathrm{TV}}\\
    		&\le2\sum_{v\in N_m}Z_v(m)\1_{\{\gamma_v(m)<\gamma_v(n)\}}
    		\le2D_{n,m}\le2D_n,
    	\end{align}
    	where $\|\eta\|_{\mathrm{TV}}=\sup_{|f|\le1}|\int f\,\d\eta|$.
    	Thus $\mu_n$ converges almost surely in total variation to a finite
    	positive measure $\mu$. By the first identity of \eqref{eq:consistency},  the total mass of $\mu_n$ is $Z_\infty$, which implies that the total mass of $\mu$ is $Z_\infty$. Each $\mu_n$ is
    	supported on the countable set
    	\begin{equation}
    	\Gamma = \bigcup_{n\ge1}\{-\gamma_u(n):u\in N_n\}.
    	\end{equation}
    	Consequently $\mu(\Gamma^c)=0$, so $\mu$ is purely atomic.
    	
    	By the branching property, it holds that
    	\begin{equation}
        Z^{-\rho}(x,n) = \lim_{t\to\infty}Z^{-\rho}(x,n,t)        =\mu_n((0,x)).
    	\end{equation}
    	Since $Z^{-\rho}(x,s)$ is non-increasing in $s$, we can take the limit along integers and then obtain that almost surely
    	\begin{equation}
    		Z^{-\rho}(x) = \lim_{n\to\infty} Z^{-\rho}(x,n) = \lim_{n\to\infty}      \mu_n((0,x)) = \mu((0,x)),
    	\end{equation}
    which implies that
    $Z^{-\rho}(\d x)=\mu.$ We finish the proof of (i).

    (ii) For any fixed $x\ge0$, the many-to-one formula and  the continuity of the distribution of $\min_{s\leq n} (B_s-\rho s)$ give $\P(x\in\Gamma)=0$. Therefore, $Z^{-\rho}(\{x\}) = 0$ almost surely.    Hence $Z^{-\rho}((0,x))$ and $Z^{-\rho}((0,x])$ agree at each such $x$ almost surely, 	which proves (ii).

   	(iii) If $\rho\le-\sqrt{2}$, the additive
    	martingale with critical parameter $-\sqrt{2}$ satisfies
    	\begin{equation}
    	\lim_{t\rightarrow\infty} \sum_{u\in N_t}e^{-\sqrt{2} X_u(t)-2t} = 0
    	\quad\P\text{-a.s.}
    	\end{equation}
    	Therefore,
    	\begin{equation}
    		\max_{u\in N_t} (-\sqrt{2}X_u(t)-2t) \to -\infty \mbox{ as } t\to\infty
    		\quad\P\text{-a.s.},
    	\end{equation}
    	and there exists a random time $T\in Z_+$ such that for any $t>T$ and $u\in N_t$,
    	\[
    	X_u(t)-\rho t\ge X_u(t)+\sqrt{2} t>0.
    	\]
    	The $\gamma_u(s)$ can not decrease after time $T$. By \eqref{eq:consistency}, $\mu_m=\mu_T$ for every integer $m\ge T$,
    	and hence $Z^{-\rho}(\d x)=\mu_T(\d x)$ has finitely many atoms $\P$-almost surely.
    	
    	Finally, suppose $-\sqrt{2}<\rho<\sqrt{2}$. If $X_u(n)-\rho n<-x$, by \eqref{eq:consistency} and the definition of $\mu_n$,
    	\begin{equation}\label{eq:positive-tail}
    		Z^{-\rho}((x,\infty))\ge \mu_n((x,\infty)) \ge Z_u(n)>0.
    	\end{equation}
    	Therefore,
    	\begin{equation}
    	\P\left(Z^{-\rho}((x,\infty))=0\right)
    	\le\P\left(\min_{u\in N_n} X_u(n)-\rho n>-x\right) \rightarrow 0 \quad\mbox{ as } n\rightarrow\infty.
    	\end{equation}
    	Taking a countable intersection over positive integer $x$ shows that
    	$Z^{-\rho}((x,\infty))>0$ simultaneously for every $x>0$, $\P$-almost
    	surely. Hence, the support of $Z^{-\rho}(\d x)$ is unbounded. By (i), it has countably infinitely many atoms. This completes the proof of (iii).
    \end{proof}

    Now, we prove Corollary \ref{cor_extremal}.
    \begin{proof}[Proof of Corollary \ref{cor_extremal}]
    (i)	Theorem \ref{thrm1} yields that for any
    	$f\in C_c^+(\R\times\R_+)$,
    	\begin{align}\label{eq:extremal_conver_DPPP}
    		&\lim_{t\rightarrow\infty} \E \left( e^{ -\sum_{u\in N_t} f(X_u(t)-m_t, -\gamma_u(t)) } \right)\\
    		&=  \E \left( \exp\left\{-\int_{y\in\R, z\in \R_+} \sqrt{2}C_* e^{-\sqrt{2}y} \E\left(1-e^{-\left\langle f(\cdot+y,\cdot+z),\, \widetilde{\mathrm{D}}^{\sqrt{2}} \right\rangle }\right) Z^{-\rho}(\d z)\d y \right\} \right),
    	\end{align}
    	where $\widetilde{\mathrm{D}}^{\sqrt{2}}$ is a random measure with law $\widetilde{\D}^{\sqrt{2}}$, and $\langle f,\mu \rangle$ denotes the integral of $f$ with respect to $\mu$ for some measurable function $f$ and $\sigma$-finite measure $\mu$.
    	For any continuous function $g:\R\to\R_+$ with compact support, take
    	\begin{equation}
    		f(y,z) = g(y) \1_{(0,x)}(z).
    	\end{equation}
    	Fix $x>0$, and write
    	\[
    	\Xi_t:=\sum_{u\in N_t}
    	\delta_{(X_u(t)-m_t,-\gamma_u(t))}.
    	\]
    	Let $\Xi$ denote its limit in Theorem \ref{thrm1}, with the point measures regarded as measures on
    	$\R\times\R_+$ equipped with the vague topology. By Proposition \eqref{prop:proportion} (ii),
    	\begin{equation}
    	    Z^{-\rho}(\{x\})=0
    	\quad \P\text{-a.s}
    	\end{equation}
    	Conditional on $Z^{-\rho}(\d x)$, the Poisson process of cluster centers defining $\Xi$ has zero intensity on
    	$[-m,m]\times\{0,x\}$ for every positive integer $m$. Since the decorations do not change the second coordinate, it follows that
    	\[
    \Xi\bigl(\mathbb R\times\{x\}\bigr)=0
    	\quad\P\text{-a.s}
    	\]
    The discontinuity set of $f$ is contained in
    $\operatorname{supp}(g)\times\{x\}$.
    Thus, by the Portmanteau theorem for vague convergence, the map
    	$\nu\mapsto\langle f,\nu\rangle$ is continuous at $\Xi$ almost surely. The continuous mapping theorem and the boundedness of the exponential function on $(-\infty,0]$ therefore yield
    	\[
    	\lim_{t\to\infty}\mathbb E\left[e^{-\langle f,\,\Xi_t\rangle}\right]
    	= \mathbb E\left[e^{-\langle f,\,\Xi\rangle}\right].
    	\]
    	Then, it follows from \eqref{eq:extremal_conver_DPPP} and the definition of $\widetilde{\D}^{\sqrt{2}}$ that
    	\begin{align}\label{eq:extremal_BBMAB_Lapalce}
            &\lim_{t\rightarrow\infty} \E \left( e^{ -\sum_{u\in N_t} g(X_u(t)-m_t) \1_{\{\gamma_u(t)>-x\}}} \right)\\
    		&=  \E \left( \!\exp\left\{\!-\!\int_{y\in\R, z\in \R_+}\!\!\! \sqrt{2}C_* e^{-\sqrt{2}y} \E\left(1-e^{-\left\langle g(\cdot+y),\, \mathrm{D}^{\sqrt{2}} \right\rangle \1_{\{z\in (0,x)\}} }\right) Z^{-\rho}(\d z)\d y \right\} \!\right)\\
    		&= \E \left( \exp\left\{-\int_{z\in(0,x)}Z^{-\rho}(\d z) \int_{y\in\R} \sqrt{2}C_* e^{-\sqrt{2}y} \E\left(1-e^{-\left\langle g(\cdot+y),\, \mathrm{D}^{\sqrt{2}} \right\rangle }\right) \d y \right\} \right)\\
    		&=  \E\! \left(\! \exp\left\{ \!\!-\!\int_{y\in\R}\!\!\!\sqrt{2}C_* Z^{-\rho}(x) e^{-\sqrt{2}y} \E\left(1-e^{-\left\langle g(\cdot+y),\, \mathrm{D}^{\sqrt{2}} \right\rangle }\right) \d y \right\} \!\right),
    	\end{align}
    	where $\mathrm{D}^{\sqrt{2}}$ is a random  measure with law $\D^{\sqrt{2}}$.
    This proves (i).

    (ii) It follows from Proposition \ref{prop:proportion} that
    	\begin{equation}
    		\lim_{x\rightarrow\infty} Z^{-\rho}(x) = Z_{\infty}\quad  \mbox{ $\P$-a.s.}
    	\end{equation}
    	Then, letting $x\to\infty$  on the both side of \eqref{eq:extremal_BBMAB_Lapalce}, we get that
    	\begin{align}
    		&\lim_{x\rightarrow\infty}  		\lim_{t\rightarrow\infty} \E \left( e^{ -\sum_{u\in N_t} g(X_u(t)-m_t) \1_{\{-\gamma_u(t) < x\}}} \right)\\
    		= &\, \E \left( \exp\left\{ -\int_{y\in\R} \sqrt{2}C_* Z_{\infty} e^{-\sqrt{2}y} \E\left(1-e^{-\left\langle g(\cdot+y),\, \mathrm{D}^{\sqrt{2}} \right\rangle }\right) \d y \right\} \right).
    	\end{align}
    This completes the proof of (ii).
    \end{proof}

	\appendix
	\section{Proof of Lemma \ref{lemma:many-to-one} $\mathrm{(i)}$}\label{sec:appendix_manytoone}

    \begin{proof}[Proof of Lemma \ref{lemma:many-to-one} $\mathrm{(i)}$]
    We only prove \eqref{many-to-one}  for $F$ support by
    $[0, a)\times \R$
    for some $a>0$ since the general case then follows from the monotone convergence theorem.

    Note that   $(\{(t, B_t)\}_{t\geq 0}, \bP_{r,x})$ is a time-nonhomogeneous strong Markov process, where $\bP_{r,x}$ is the law of $(t, B_t)$ starting from $(r,x)\in[0,\infty) \times \R$. Now we consider a branching Markov process as introduced at the beginning of the paper with  $(\{B(t)\}_{ t\geq 0}, \mathbf{P})$ replaced by $(\{(t, B_t)\}_{t\geq 0}, \bP_{r,x})$,  and other rules unchanged. We use $\P_{r,x}$ to denote  the law of branching Markov process starting from one particle at $(r,x)$. Set
    $$
    Y:=\sum_{u\in N_D} F(\tau_D(u), X_u(\tau_D(u)))\mathbf{1}_{(\tau_D(u)<\infty)}
    $$ and
	\begin{equation}
		u(r,x) := \P_{r,x}\left(\sum_{u\in N_D} F(\tau_D(u), X_u(\tau_D(u)))\mathbf{1}_{(\tau_D(u)<\infty)}\right).
	\end{equation}
    For $\lambda\geq 0$, define
	\begin{align}\label{def_v(r,x)}
		v_{\lambda}(r,x) = -\log \P_{r,x} e^{-\lambda Y}.
	\end{align}
	It follows from \cite[Theorem 1.4]{Dynkin01} that
	$v_{\lambda}(r,x)$ satisfies the equation
    \begin{equation}\label{eq:v_integral_equation}
	    e^{-v_{\lambda}(r,x)} = \bP_{r,x} \left[ \int_r^{\tau_D} \phi(e^{-v_{\lambda}(s,B_s)}) \d s + e^{-\lambda F(\tau_D, B_{\tau_D})\mathbf{1}_{(\tau_D<\infty)}},\right] \, x\in D,
    \end{equation}
	where $\phi(z) = \sum_{n=0}^{\infty} p_n z^n - z$.
	Note that $\phi'(1) = \sum_{n=0}^{\infty} np_n-1 = 1$. Differentiating both sides of \eqref{eq:v_integral_equation} with respect to $\lambda$, we get that
	\begin{align}
		&\P_{r,x}\left(Y e^{-\lambda Y}\right)\\
		=& \bP_{r,x} \left[ F(\tau_D, B_{\tau_D}) e^{-\lambda F(\tau_D, B_{\tau_D})}\mathbf{1}_{(\tau_D<\infty)} + \int_r^{\tau_D} \phi'(e^{-v_{\lambda}(s,B_s)})\times\P_{s,B_s}\left(Y e^{-\lambda Y}\right) \d s\right].
	\end{align}
	Then, letting $\lambda = 0$, we get
	$u(r,x)$ satisfies the following Schr\"{o}dinger equation:
	\begin{equation}\label{eq:u_integral}
		u(r,x) = \bP_{r,x} \left[ F(\tau_D, B_{\tau_D})\mathbf{1}_{(\tau_D<\infty)} + \int_r^{\tau_D} u(s,B_s) \d s\right],\quad x\in D.
	\end{equation}
    By \cite[Lemma A.I.1.5]{Dynkin93}
    the solution to  above Schr\"{o}dinger is unique and given by
    \begin{equation}\label{FK}
        u(r,x)=\bP_{r,x} \left[ e^{\tau_D-r}  F(\tau_D,B_{\tau_D}), \tau_D<\infty \right].
    \end{equation}
    Letting $r=0$, we get \eqref{many-to-one}.

    \end{proof}

    \section{Proof of Lemma \ref{lemma:technical}]}\label{sec:appendix}
	\begin{proof}[Proof of Lemma \ref{lemma:technical}]
		For any $a,b\in \R$, we define
				\begin{align}
            {\tau}_{a}^{b} := \inf\{s\geq 0: a+B_s \leq b s \}. \label{def_tau_rho}
		\end{align}
		For any $u\in N_t$, define
		\begin{equation}\label{def_tau_u}
            {\tau}_{a}^{b}(u): = \inf\{s\in [0,t]: a+X_u(s) \leq b s \}.
		\end{equation}
        Then for any fixed $p\in(0, 1)$, we have
		\begin{align}
			&\hspace{1.3em} \P \left(\exists u\in N_t: x+X_u(t) \geq m_t+y, \inf_{r\in [0,t]} (x+X_u(r)-\rho r)<0 \right)\\
			&\leq \P \left(\exists u\in N_t: x+X_u(t) \geq m_t+y, {\tau}_{x}^{\rho}(u) \in [0,pt] \right)\\
			&\hspace{1.3em}+ \P \left(\exists u\in N_t: x+X_u(t) \geq m_t+y, {\tau}_{x}^{\rho}(u)\in (pt,t] \right)\\
			&=: II_1 + II_2.
		\end{align}
		For the second term, it follows from \eqref{many-to-one'} and the Markov property that
		\begin{align}
			II_2 =&\,\P \left(\exists u\in N_t: x+X_u(t) \geq m_t+y, {\tau}_{x}^{\rho}(u)\in (pt,t] \right) \\
			\leq&\, \E\left[ \sum_{u\in N_t} \1_{\left\{ {\tau}_{x}^{\rho}(u)\in (pt,t],\, x+X_u(t)\geq m_t + y \right\} } \right]\\
			\leq &\, \bE \left[e^t\1_{\{ {\tau}_{x}^{\rho}\in (pt,t], \, x+B_t\geq m_t+y \}}\right]\\
			=&\, \bE\left[ e^t 1_{\{{\tau}_{x}^{\rho}\in (pt,t]\} } \bP(x+z+B_{t-r}\geq m_t+y)\mid_{z=B_r,r=\underline{\tau}_{x}^{\rho}} \right]\\
			=&\, e^t\int_{pt}^t  \bP({\tau}_{x}^{\rho}\in\d r)  \bP\left(B_{t-r}\geq \sqrt{2}(t-r) + (\sqrt{2}-\rho)r -\frac{3}{2\sqrt{2}}\log t + y\right),
		\end{align}
        where for the last equality we used the fact that $x+B_{{\tau}_{x}^{\rho}} = \rho r$       when ${\tau}_{x}^{\rho} = r$.
		For any $z\in\R$, an application of Markov's inequality shows that
		\begin{equation}
			e^{t-r} \bP(B_{t-r} \geq \sqrt{2}(t-r) + z) \leq e^{-\sqrt{2}z}.
		\end{equation}
		It follows from \cite[Section 3.5.C]{Karatzas} that
		\begin{equation}\label{density-tau}
			\bP({\tau}_{x}^{\rho}\in \d r) = \frac{x}{\sqrt{2\pi r^3}} \exp\left\{ - \frac{(x-\rho r)^2}{2r} \right\}.
		\end{equation}
		Therefore,
        \begin{align}\label{eq:estimate_II_2}
        	II_2 &\leq \int_{pt}^t e^r \frac{x}{\sqrt{2\pi r^3}} e^{ - \frac{(x-\rho r)^2}{2r} } e^{-(2-\sqrt{2}\rho)r + \frac{3}{2} \log t - \sqrt{2}y}\d r\\
        	&= \int_{pt}^t \frac{x}{\sqrt{2\pi }} \frac{t^\frac{3}{2}}{r^\frac{3}{2}} e^{-(\frac{\rho}{\sqrt{2}}-1)^2r - \frac{x^2}{2r} + x\rho - \sqrt{2}y}\d r\\
        	&\leq
        \frac{1}{\sqrt{2\pi }}
        e^{-\sqrt{2}y} x
        e^{x\rho}
        \int_{pt}^{t} \frac{1}{p^{\frac{3}{2}}} e^{-(\frac{\rho}{\sqrt{2}}-1)^2r} \d r \rightarrow 0, \mbox{ as } t\rightarrow\infty.
        \end{align}
	
	    For any $a,b\in \R$, define
	    \begin{equation}
          {D}_a^b:=\{(t,y):t\geq 0, a+y > bt\}.
	    \end{equation}
        Recall the definition  of the stopping line given by \eqref{def_stopping_line} and of the set of particles stopped on the line given by \eqref{def_stopping_set}. Note that
	    \begin{align}
	    	II_1 &= \P\left( \exists u\in N_t: {\tau}_{x}^{\rho}(u)\in [0,pt], x+X_u(t) \geq m_t +y \right)\\
	    	& = \P\left( \exists v\in N_{{D}_x^{\rho}}: {\tau}_x^{\rho}(v)\in [0,pt],\, \exists u\succ v, u\in N_t, x+X_u(t) \geq m_t+y \right)
	    \end{align}
	    and this probability is less than or equal to the expectation of the number of $v\in N_{{D}_x^{\rho}}$ which satisfies ${\tau}_x^{\rho}(v)\in [0,pt]$ and $\exists u\succ v, u\in N_t, X_u(t) \geq m_t+y$, we have the following estimate
	    \begin{align}
	    	II_1 &\leq \E \bigg{[}\sum_{v\in N_{{D}_x^{\rho}}} \1_{\{{\tau}_x^{\rho}(v)\in [0,pt],\, \exists u\succ v, u\in N_t, x+X_u(t) \geq m_t+y \}}\bigg{]}\\
	    	&\leq \E \Bigg{[}\E\bigg{(} \sum_{v\in N_{{D}_{x}^{\rho}}} \1_{\{{\tau}_x^{\rho}(v)\in [0,pt],\, \exists u\succ v, u\in N_t, x+X_u(t) \geq m_t+y \}} |	    	\F_{\tau_{x}^{\rho}}\bigg{)}\Bigg{]}\\
	    	&= \E \Bigg{[}\sum_{v\in N_{{D}_{x}^{\rho}}} \1_{\{{\tau}_x^{\rho}(v)\in [0,pt]\}} \P(x+z+M_{t-r}(v) \geq m_t+y) \mid_{z = X_v(r), r={\tau}_x^{\rho}(v)} \Bigg{]},
	    \end{align}
	    where $\F_{{\tau}_{x}^{\rho}}$ is the  filtration
	    generated by the spatial paths and the number of offspring of the individuals before hitting the stopping line $L_{{D}_{x}^{\rho}}$ and $M_{t-{\tau}_x^{\rho}(v)}(v) = \max_{u\succ v, u\in N_t} X_u(t)$. By the  many-to-one formula \eqref{many-to-one} and Lemma \ref{lemma:M_tail}, we have
	    \begin{align}
	    	 II_1 &\leq \int_0^{pt} e^r \bP({\tau}_{x}^{\rho}\in \d r) \P(\rho r + M_{t-r}(v)\geq m_t+y)\\
	    	&= \int_0^{pt} e^r \bP({\tau}_{x}^{\rho}\in \d r) \P(M_{t-r}(v)\geq m_{t-r} + \sqrt{2}r - \frac{3}{2\sqrt{2}}\log \frac{t}{t-r} + y -\rho r)\\
	    	&\leq \int_0^{pt} e^r \frac{x}{\sqrt{2\pi r^3}} e^{ -\frac{(x-\rho r)^2}{2r} } b (1+\sqrt{2}r-\rho r+y) e^{-\sqrt{2}\left( \sqrt{2}r - \rho r - \frac{3}{2\sqrt{2}}\log \frac{t}{t-r} + y \right)+1} \d r\\
	    	&\leq Ce^{-\sqrt{2}y} \int_0^{pt} \frac{x}{\sqrt{2\pi r^3}} (1+\sqrt{2}r-\rho r+y) \left( \frac{t}{(1-p)t} \right)^{\frac{3}{2}} e^{-\frac{(x-(\sqrt{2}-\rho)r)^2}{2r}+(2\rho-\sqrt{2})x} \d r,
        \end{align}
    where $C$ is a positive constant. Using \eqref{density-tau}, we get
    \begin{align}
	    II_1 &\leq Ce^{-\sqrt{2}y+(2\rho-\sqrt{2})x} \left[ (\sqrt{2}-\rho)\bE\left({\tau}_{x}^{\sqrt{2}-\rho} \1_{\{{\tau}_{x}^{\sqrt{2}-\rho} \leq pt \}}\right) + (1+y) \bE\left( \1_{\{{\tau}_{x}^{\sqrt{2}-\rho} \leq pt \}}\right) \right]\\
        &\leq C (1+x+y) e^{-\sqrt{2}y+(2\rho-\sqrt{2})x}, \label{eq:estimate_II_1}
	\end{align}
	    where the last inequality relies on the identity $\bE[\underline{\tau}_{x}^{\sqrt{2}-\rho}] = \frac{x}{\sqrt{2}-\rho}$, derived in \cite[Exercise 3.5.10]{Karatzas}. Combining \eqref{eq:estimate_II_2} and \eqref{eq:estimate_II_1}, we complete the proof.
	\end{proof}
	
	\vspace{.1in}
	\textbf{Acknowledgment}:
		We thank Bastien Mallein for raising this question and proposing the approach of studying the convergence of the joint point process at the conference on branching processes in Shenzhen MSU-BIT University. Yan-Xia Ren  is supported by  NSFC (Grant Nos. 12231002 and 12631006) and the Fundamental Research Funds for Central Universities, Peking University LMEQF. Renming Song's research  was supported in part by a grant from the Simons Foundation (\#960480, Renming Song).
		Fan Yang is supported by the Fundamental Research Funds for the Central Universities.

\end{document}